\documentclass[11pt]{amsart}
\usepackage{float}
\usepackage{graphicx}
\usepackage{amssymb}
\theoremstyle{plain}
\newtheorem{thm}{Theorem}[section]
\newtheorem{lemma}[thm]{Lemma}

\theoremstyle{definition}

\def\dim{\mathop{\hbox {dim}}\nolimits}

\newcommand{\frg}{\mathfrak{g}}
\newcommand{\frh}{\mathfrak{h}}

\newcommand{\bbC}{\mathbb{C}}

\newcommand{\bbN}{\mathbb{N}}

\newcommand{\bbR}{\mathbb{R}}

\newcommand{\bbZ}{\mathbb{Z}}

\begin{document}

\title[Panyushev conjecture]
{The reverse operator orbits on $\Delta(1)$ and a conjecture of
Panyushev}

%

\author{Chao-Ping Dong}

\address[Dong]
{Institute of Mathematics,  Hunan University,
Changsha 410082, China}
\email{chaoping@hnu.edu.cn}

\abstract{We verify conjecture 5.11 of Panyushev  [\emph{Antichains in weight posets associated with gradings of simple
Lie algebras}, Math Z 281(3):1191--1214, 2015].}
\endabstract



\maketitle


\section{Introduction}

Let $\frg$ be a finite-dimensional simple Lie algebra over $\bbC$. Fix a Cartan subalgebra $\frh$ of $\frg$.
The associated root system is $\Delta=\Delta(\frg, \frh)\subseteq\frh_{\bbR}^*$. Recall that a decomposition
\begin{equation}\label{grading}
\frg=\bigoplus_{i\in \bbZ}\frg(i)
\end{equation}
is  a \emph{$\bbZ$-grading} of $\frg$ if $[\frg(i), \frg(j)]\subseteq \frg(i+j)$ for any $i, j\in\bbZ$.
In particular, in such a case, $\frg(0)$ is a Lie subalgebra of $\frg$. Since each derivation of $\frg$ is inner, there exists $h_0\in\frg(0)$ such that $\frg(i)=\{x\in\frg\mid [h_0, x]=i x\}$. The element $h_0$ is said to be \emph{defining} for the grading \eqref{grading}. Without loss of generality, one may assume that $h_0\in\frh$. Then $\frh\subseteq\frg(0)$. Let $\Delta(i)$ be the set of roots in $\frg(i)$. Then we can
choose a set of positive roots $\Delta(0)^+$ for $\Delta(0)$ such that
$$
\Delta^+ :=\Delta(0)^+\sqcup \Delta(1)\sqcup \Delta(2)\sqcup \cdots
$$
is a set of positive roots of $\Delta(\frg, \frh)$. Let $\Pi$ be the
corresponding simple roots, and put $\Pi(i)=\Delta(i)\cap \Pi$. Note
that the grading \eqref{grading} is fully determined by
$\Pi=\bigsqcup_{i\geq 0} \Pi(i)$. We refer the reader to Ch.~3, \S 3
of \cite{GOV} for generalities on gradings of Lie algebras. Each
$\Delta(i)$, $i\geq 1$, inherits a poset structure from the usual
one of $\Delta^+$. That is, let $\alpha$ and $\beta$ be two roots of
$\Delta(i)$, then $\beta\geq\alpha$ if and only if $\beta-\alpha$ is
a nonnegative integer combination of simple roots.

Recently, Panyushev initiated the study of the rich structure of
$\Delta(1)$ in \cite{P}. In particular,  he raised five
conjectures concerning the $\mathcal{M}$-polynomial,
$\mathcal{N}$-polynomial and the reverse operator of $\Delta(1)$.
Note that Conjectures 5.1, 5.2 and 5.12 there have been solved by
Weng and the author \cite{DW}. The current paper aims
to handle conjecture 5.11 of \cite{P}. Let us prepare more notation.

Recall that a subset $I$ of a finite poset $(P, \leq)$ is a
\emph{lower} (resp., \emph{upper}) \emph{ideal} if $x\leq y$ in $P$
and $y\in I$ (resp. $x\in I$) implies that $x\in I$ (resp. $y\in
I$). We collect the  lower ideals of $P$ as $J(P)$, which is
partially ordered by inclusion. A subset $A$ of $(P, \leq)$ is an
\emph{antichain} if any two elements in $A$ are non-comparable under
$\leq$. We collect the antichains of $P$ as $\mathrm{An}(P)$. For
any $x\in P$, let $I_{\leq x}=\{y\in P\mid y\leq x\}$. Given an
antichain $A$ of $P$, let $I(A)=\bigcup_{a\in A} I_{\leq a}$. The
\emph{reverse operator} $\mathfrak{X}$ is defined by
$\mathfrak{X}(A)=\min (P\setminus I(A))$. Since antichains of $P$
are in bijection with lower (resp. upper) ideals of $P$, the reverse
operator acts on lower (resp. upper) ideals of $P$ as well. Note
that the current $\mathfrak{X}$ is inverse to the reverse operator
$\mathfrak{X}^{\prime}$ in Definition 1 of \cite{P}, see Lemma
\ref{lemma-inverse-reverse-operator}. Thus replacing $\mathfrak{X}^{\prime}$ by $\mathfrak{X}$ does not affect our
forthcoming discussion on orbits.

We say the $\bbZ$-grading \eqref{grading} is \emph{extra-special} if
\begin{equation}\label{extra-special}
\frg=\frg(-2)\oplus \frg(-1) \oplus \frg(0) \oplus \frg(1)
\oplus \frg(2) \mbox{  and  }\dim\frg(2)=1,
\end{equation}
Up to conjugation, any simple Lie algebra $\frg$ has a unique extra-special $\bbZ$-grading. Without loss of generality, we assume that $\Delta(2)=\{\theta\}$ , where $\theta$ is the highest root of $\Delta^+$.
Namely, we may assume that the grading \eqref{extra-special} is defined by the element $\theta^{\vee}$, the dual root of $\theta$. In such a case, we have
\begin{equation}\label{Delta-one}
\Delta(1)=\{\alpha\in\Delta^+\mid (\alpha, \theta^{\vee})=1\}.
\end{equation}
Let $\mathrm{ht}$ be the height function. Recall that $h:=\mathrm{ht}(\theta)+1$ is the \emph{Coxeter number}
of $\Delta$. Let $h^*$ be the \emph{dual Coxeter number }of
$\Delta$. That is, $h^*$ is the height of
$\theta^{\vee}$ in $\Delta^{\vee}$. As noted on p.~1203 of \cite{P},
we have $|\Delta(1)|=2h^*-4$. We call a lower (resp. upper) ideal
$I$ of $\Delta(1)$ \emph{Lagrangian} if $|I|=h^*-2$. Write
$\Delta_l$ (resp. $\Pi_l$) for the set of \emph{all} (resp.
\emph{simple}) \emph{long} roots. In the simply-laced cases, all
roots are assumed to be both long and short. Note that $\theta$ is
always long, while $\theta^{\vee}$ is always short.

Now Conjecture 5.11 of \cite{P} is stated as follows.

\medskip
\noindent \textbf{Panyushev conjecture.}\quad In any extra-special
$\bbZ$-grading of $\frg$, the number of
$\mathfrak{X}_{\Delta(1)}$-orbits equals $|\Pi_l|$, and each orbit
is of size $h-1$. Furthermore, if $h$ is even (which only excludes the case $A_{2k}$ where $h=2k+1$), then each
$\mathfrak{X}_{\Delta(1)}$-orbit contains a unique Lagrangian lower
ideal.

\medskip

Originally, the conjecture is stated in terms of upper ideals and
the reverse operator $\mathfrak{X}^{\prime}$. One agrees that we can
equivalently phrase it using lower ideals and $\mathfrak{X}$. The
main result of the current paper is the following.

\begin{thm}\label{thm-main}
Panyushev conjecture is true.
\end{thm}

After collecting necessary preliminaries in Section 2, the above
theorem will be proven in Section 3. Moreover, we note that by our
calculations in Section 3, one checks easily that for any
extra-special $1$-standard $\bbZ$-grading of $\frg$, all the
statements of Conjecture 5.3 in \cite{P} hold.

\medskip

\noindent\textbf{Notation.} Let $\bbN =\{0, 1, 2, \dots\}$, and let
$\mathbb{P}=\{1, 2, \dots\}$. For each $n\in\mathbb{P}$, $[n]$
denotes the poset $(\{1, 2, \dots, n\}, \leq)$.

\section{Preliminaries}

Let us collect some preliminary results in this section. Firstly,
let us compare the two reverse operators. Let $(P, \leq)$ be any
finite poset. For any $x\in P$, let $I_{\geq x}=\{y\in P\mid y\geq
x\}$. For any antichain $A$ of $P$, put $I_{+}(A)=\bigcup_{a\in A}
I_{\geq a}$. Recall that in Definition 1 of \cite{P}, the reverse
operator $\mathfrak{X}^{\prime}$ is given by
$\mathfrak{X}^{\prime}(A)=\max (P\setminus I_{+}(A))$.

\begin{lemma}\label{lemma-inverse-reverse-operator}
The operators $\mathfrak{X}$ and  $\mathfrak{X}^{\prime}$ are
inverse to each other.
\end{lemma}
\begin{proof}
Take any antichain $A$ of $P$, note that
$$I_{+}(\min(P\setminus
I(A)))=P\setminus I(A)\mbox{ and } I(\max(P\setminus
I_{+}(A)))=P\setminus I_{+}(A).
$$
Then the lemma follows.
\end{proof}

Let $(P_i,\leq), i=1, 2$ be two finite posets. One can define a
poset structure on $P_1\times P_2$ by setting $(u_1, v_1)\leq (u_2,
v_2)$ if and only if $u_1\leq u_2$ in $P_1$ and $v_1\leq v_2$ in
$P_2$. We simply denote the resulting poset by $P_1 \times P_2$. The
following well-known lemma describes the lower ideals of
$[m]\times P$.

\begin{lemma}\label{lemma-ideals-CnP}
Let $P$ be a finite poset. Let $I$ be a subset of $[m]\times P$. For
$1\leq i\leq m$, denote $I_i=\{a\in P\mid (i, a)\in I\}$. Then $I$
is a lower ideal of $[m]\times P$ if and only if each $I_i$ is a
lower ideal of $P$, and $I_m\subseteq I_{m-1}\subseteq \cdots
\subseteq I_{1}$.
\end{lemma}

In this section, by a \emph{finite graded poset} we always mean a
finite poset $P$ with a rank function $r$ from  $P$ to the positive
integers $\mathbb{P}$ such that all the minimal elements  have rank
$1$, and $r(x)=r(y)+1$ if $x$ covers $y$. In such a case, let $P_i$
be the set of elements in $P$ with rank $i$. The sets $P_i$ are said
to be the \emph{rank levels} of $P$. Suppose that
$P=\bigsqcup_{j=1}^{d} P_j$. Let $P_0$ be the empty set $\emptyset$.
Put $L_i=\bigsqcup_{j=1}^{i} P_j$ for $1\leq j\leq d$, and let $L_0$ be the empty set.
We call those $L_i$ \emph{rank level lower ideals}.

Let $\mathfrak{X}$  be the reverse operator on $[m]\times P$. In
view of Lemma \ref{lemma-ideals-CnP}, we denote by $(I_1, \cdots,
I_m)$ a general lower ideal of $[m]\times P$, where each $I_i\in
J(P)$ and $I_m\subseteq  \cdots \subseteq I_{1}$. We say that the
lower ideal $(I_1, \cdots, I_m)$ is \emph{full rank} if each $I_i$
is a rank level lower ideal of $P$. Let $\mathcal{O}(I_1, \cdots,
I_m)$ be the $\mathfrak{X}_{[m]\times P}$-orbit of $(I_1, \cdots,
I_m)$. The following lemma will be helpful in determining
$\mathfrak{X}_{[m]\times P}$-orbits  consisting of rank level lower
ideals.

\begin{lemma}\label{lemma-operator-ideals-CmP}
Keep the notation as above.  Then
for any $n_0\in \bbN$, $n_i\in\mathbb{P}$ ($1\leq i\leq s$) such that $\sum_{i=0}^{s} n_i =m$, we have
\begin{equation}\label{rank-level}
\mathfrak{X}_{[m]\times P}(L_d^{n_0}, L_{i_1}^{n_1}, \cdots, L_{i_s}^{n_s})=
(L_{i_1+1}^{n_0+1}, L_{i_2+1}^{n_1}, \cdots, L_{i_s+1}^{n_{s-1}}, L_0^{n_s-1}),
\end{equation}
where $0\leq i_s<\cdots <i_1<d$, $L_d^{n_0}$ denotes $n_0$ copies of $L_d$ and so on.

\end{lemma}
\begin{proof}
Note that under the above assumptions, $(L_d^{n_0}, L_{i_1}^{n_1}, \cdots, L_{i_s}^{n_s})$ is a lower ideal of $[m]\times P$ in view of Lemma \ref{lemma-ideals-CnP}. Then analyzing the minimal elements of $([m]\times P)\setminus (L_d^{n_0}, L_{i_1}^{n_1}, \cdots, L_{i_s}^{n_s})$ leads one to \eqref{rank-level}.

\end{proof}

\begin{lemma}\label{lemma-operator-types}
Let $(I_1, \cdots, I_m)$ be an arbitrary lower ideal of $[m]\times P$.
Then $(I_1, \cdots, I_m)$ is full rank if and only if each lower ideal
in the orbit $\mathcal{O}(I_1, \cdots, I_m)$ is full rank.
\end{lemma}
\begin{proof}
Use Lemma \ref{lemma-operator-ideals-CmP}.
\end{proof}

The above lemma tells us that there are two types of $\mathfrak{X}$-orbits: in the first type each lower ideal is full rank, while in the second type each lower ideal is not. We call them \emph{type I} and \emph{type II}, respectively.

For any $n\geq 2$, let $K_{n-1}=[n-1]\oplus([1]\sqcup [1])\oplus [n-1]$ (the ordinal sum, see
p.~246 of \cite{St}).  We label the elements of $K_{n-1}$ by $1$, $2$, $\cdots$,
$n-1$, $n$, $n^{\prime}$, $n+1$, $\cdots$, $2n-2$, $2n-1$. Figure 1
illustrates the labeling for the Hasse diagram of $K_3$. Note that $L_i$ ($0\leq i\leq 2n-1$) are all the full rank lower ideals. For instance, we have $L_{n}=\{1, 2, \cdots, n, n^{\prime}\}$.
Moreover, we put $I_{n}=\{1,  \cdots, n-1,
n\}$ and $I_{n^{\prime}}=\{1,  \cdots, n-1, n^{\prime}\}$. The following lemma will be helpful in analyzing the  $\mathfrak{X}_{[m]\times K_{n-1}}$-orbits of type II.

\begin{lemma}\label{lemma-operator-ideals-CmK}
Fix $n_0\in \bbN$, $n_i\in\mathbb{P}$ ($1\leq i\leq s$),
$m_j\in\mathbb{P}$ ($0\leq j\leq t$) such that $\sum_{i=0}^{s} n_i +
\sum_{j=0}^{t} m_j=m$. Take any $0\leq j_t< \cdots<j_1<n\leq
i_s<\cdots <i_1<2n-1$, we have
\begin{align*}
\mathfrak{X}_{[m]\times K_{n-1}}&(L_{2n-1}^{n_0}, L_{i_1}^{n_1}, \cdots, L_{i_s}^{n_s}, I_n^{m_0}, L_{j_1}^{m_1}, \cdots, L_{j_t}^{m_t})=\\
&\begin{cases}
( L_{i_1+1}^{n_0+1}, L_{i_2+1}^{n_1}, \cdots, L_{i_s+1}^{n_{s-1}}, I_{n^{\prime}}^{n_s}, L_{j_1+1}^{m_0},
L_{j_2+1}^{m_1}, \cdots, L_{j_t+1}^{m_{t-1}}, L_0^{m_t-1} ) & \mbox { if } j_1 < n-1;\\
( L_{i_1+1}^{n_0+1}, L_{i_2+1}^{n_1}, \cdots, L_{i_s+1}^{n_{s-1}}, L_{n}^{n_s}, I_n^{m_0},
\, \, \, \, L_{j_2+1}^{m_1}, \cdots, L_{j_t+1}^{m_{t-1}}, L_0^{m_t-1} )& \mbox { if } j_1 = n-1.
\end{cases}
\end{align*}
\end{lemma}
\begin{proof}
Analyzing the minimal elements of $$([m]\times K_{n-1})\setminus (L_{2n-1}^{n_0}, L_{i_1}^{n_1}, \cdots, L_{i_s}^{n_s}, I_n^{m_0}, L_{j_1}^{m_1}, \cdots, L_{j_t}^{m_t})$$ leads one to the desired expression.
\end{proof}

\begin{figure}[]
\centering \scalebox{0.4}{\includegraphics{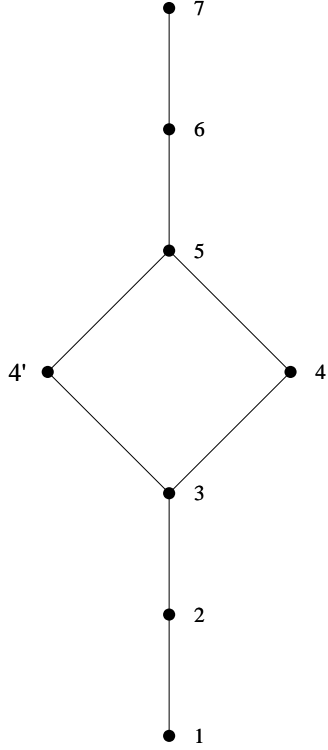}}
\caption{The labeled Hasse diagram of $K_3$}
\end{figure}

\section{Panyushev conjecture}

This section is devoted to proving Theorem \ref{thm-main}.

\noindent \emph{Proof of Theorem \ref{thm-main}.} Note that when
$\frg$ is $A_n$,  the extra-special $\Delta(1)\cong [n-1]\sqcup
[n-1]$; when $\frg$ is $C_n$, the extra-special $\Delta(1)\cong
[2n-2]$. One can verify Theorem \ref{thm-main} for these two cases
without much effort. We omit the details.

For $\frg=B_n$,  the extra-special $\Delta(1)= [2]\times [2n-3]$.
Now $|\Pi_{l}|=n-1$, $h-1=2n-1$, and $h^*-2=2n-3$. As in Section 2,
let $L_i$ ($0\leq i\leq 2n-3$) be the rank level lower ideals. For
simplicity, we simply denote $\mathfrak{X}_{[2]\times [2n-3]}$ by
$\mathfrak{X}$. For any $1\leq i\leq n-2$, let us analyze the type I
$\mathfrak{X}$-orbit $\mathcal{O}(L_i, L_i)$ via the aid of Lemma
\ref{lemma-operator-ideals-CmP}:
\begin{align*}
\mathfrak{X}(L_i, L_i)&=(L_{i+1}, L_0),\\
\mathfrak{X}^{2n-4-i}(L_{i+1}, L_0)&=(L_{2n-3}, L_{2n-4-i}),\\
\mathfrak{X}(L_{2n-3}, L_{2n-4-i})&=(L_{2n-3-i}, L_{2n-3-i}),\\
\mathfrak{X}(L_{2n-3-i}, L_{2n-3-i})&=(L_{2n-2-i}, L_{0}),\\
\mathfrak{X}^{i-1}(L_{2n-2-i}, L_{0})&=(L_{2n-3}, L_{i-1}),\\
\mathfrak{X}(L_{2n-3}, L_{i-1})&=(L_{i}, L_{i}).
\end{align*}
Thus $\mathcal{O}(L_i, L_i)$ consists of $2n-1$ elements. Moreover,
in this orbit,  $(L_{2n-2-\frac{i+1}{2}}, L_{\frac{i-1}{2}})$ is the
unique ideal with size $2n$ when $i$ is odd,  $(L_{n+\frac{i}{2}-1},
L_{n-\frac{i}{2}-2})$ is the unique ideal with size $2n$ when $i$ is
even. Similarly, the orbit $\mathcal{O}(L_0, L_0)$ consists of
$2n-1$ elements and contains a unique ideal with size $2n$:
$(L_{n-1}, L_{n-2})$. Since there are $(n-1)(2n-1)$ lower ideals in
$[2]\times [2n-3]$ by Lemma \ref{lemma-ideals-CnP}, one sees that
all the $\mathfrak{X}$-orbits have been exhausted, and Theorem
\ref{thm-main} holds for $B_{n}$.

Let us consider $D_{n+2}$, where the extra-special $\Delta(1)\cong
[2]\times K_{n-1}$. We adopt the notation as in Section 2. For simplicity,
we write $\mathfrak{X}_{[2]\times K_{n-1}}$ by $\mathfrak{X}$.  We propose
the following.

\textbf{Claim.} $\mathcal{O}(L_i, L_i)$, $0\leq i\leq n-1$,
 $\mathcal{O}(I_n, I_n)$, and $\mathcal{O}(I_{n^{\prime}},
I_{n^{\prime}})$ exhausts the orbits of $\mathfrak{X}$ on $[2]\times
K_{n-1}$. Moreover, each orbit has size $2n+1$ and contains a unique
lower ideal with size $2n$.

Indeed, firstly, for any $0\leq i\leq n-1$, observe that by Lemma
\ref{lemma-operator-ideals-CmP}, we have
\begin{align*}
\mathfrak{X}(L_i, L_i)&=(L_{i+1}, I_0),\\
\mathfrak{X}^{2n-i-2}(L_{i+1}, L_0)&=(L_{2n-1}, L_{2n-i-2}),\\
\mathfrak{X}(L_{2n-1}, L_{2n-i-2})&=(L_{2n-i-1}, L_{2n-i-1}),\\
\mathfrak{X}(L_{2n-i-1}, L_{2n-i-1})&=(L_{2n-i}, L_{0}),\\
\mathfrak{X}^{i-1}(L_{2n-i}, L_{0})&=(L_{2n-1}, L_{i-1}),\\
\mathfrak{X}(L_{2n-1}, L_{i-1})&=(L_{i}, L_{i}).
\end{align*}
Thus the type I orbit $\mathcal{O}(L_i, L_i)$ consists of $2n+1$ elements. Moreover,
in this orbit,  $(L_{2n-i+\frac{i-1}{2}}, L_{\frac{i-1}{2}})$ is the
unique ideal with size $2n$ when $i$ is odd,  $(L_{n+\frac{i}{2}},
L_{n-\frac{i}{2}-1})$ is the unique ideal with size $2n$ when $i>0$
is even, while $(L_{n}, L_{n-1})$ is the unique ideal
with size $2n$ when $i=0$.

Secondly, assume that $n$ is even and let us analyze the orbit
$\mathcal{O}(I_n, I_n)$. Indeed, by Lemma \ref{lemma-operator-ideals-CmK}, we have
\begin{align*}
\mathfrak{X}(I_n, I_n)&=(I_{n^{\prime}}, L_0),\\
\mathfrak{X}^{n-1}(I_{n^{\prime}}, L_0)&=(I_{n}, L_{n-1}),\\
\mathfrak{X}(I_{n}, L_{n-1})&=(L_{n}, I_{n}),\\
\mathfrak{X}^{n-1}(L_{n}, I_{n})&=(L_{2n-1}, I_{n^{\prime}}),\\
\mathfrak{X}(L_{2n-1}, I_{n^{\prime}})&=(I_{n}, I_{n}).
\end{align*}
Thus the type II orbit $\mathcal{O}(I_n, I_n)$ consists of $2n+1$ elements. Moreover,
in this orbit,  $(I_n, I_n)$ is the unique ideal with size $2n$. The
analysis of the orbit $\mathcal{O}(I_{n^{\prime}}, I_{n^{\prime}})$
is entirely similar.

Finally, assume that $n$ is odd and let us analyze the orbit
$\mathcal{O}(I_n, I_n)$. Indeed,  by Lemma \ref{lemma-operator-ideals-CmK},  we have
\begin{align*}
\mathfrak{X}(I_n, I_n)&=(I_{n^{\prime}}, L_0),\\
\mathfrak{X}^{n-1}(I_{n^{\prime}}, L_0)&=(I_{n^{\prime}}, L_{n-1}),\\
\mathfrak{X}(I_{n^{\prime}}, L_{n-1})&=(L_{n}, I_{n^{\prime}}),\\
\mathfrak{X}^{n-1}(L_{n}, I_{n^{\prime}})&=(L_{2n-1}, I_{n^{\prime}}),\\
\mathfrak{X}(L_{2n-1}, I_{n^{\prime}})&=(I_{n}, I_{n}).
\end{align*}
Thus the type II  orbit $\mathcal{O}(I_n, I_n)$ consists of $2n+1$ elements. Moreover,
in this orbit,  $(I_n, I_n)$ is the unique ideal with size $2n$. The
analysis of the orbit $\mathcal{O}(I_{n^{\prime}}, I_{n^{\prime}})$
is entirely similar.

To sum up, we have verified the claim since there are $(n+2)(2n+1)$
lower ideals in $[2]\times K_{n-1}$ by Lemma \ref{lemma-ideals-CnP}.
Note that $|\Pi_{l}|=n+2$, $h=h^*=2n+2$ for $\frg=D_{n+2}$, one sees that Theorem
\ref{thm-main} holds for $D_{n+2}$.

Theorem \ref{thm-main} has been verified for all exceptional Lie
algebras using \texttt{Mathematica}. We only present the details for
$E_6$, where $\Delta(1)=[\alpha_2]$, and the Dynkin diagram is as follows.

\begin{figure}[H]
\centering \scalebox{0.5}{\includegraphics{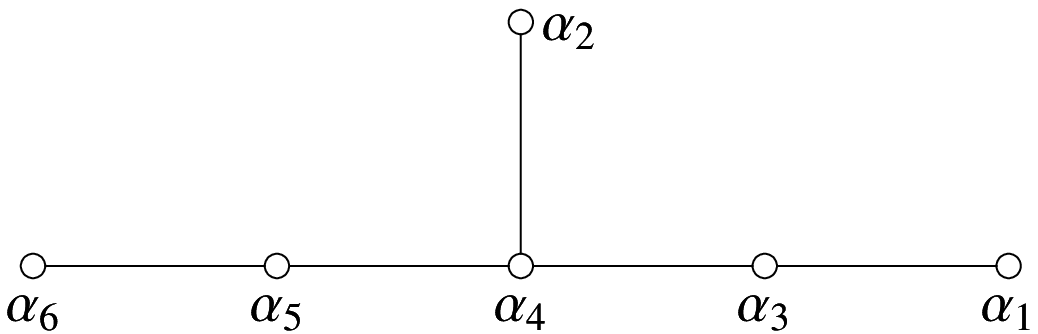}}
\end{figure}

Note that $|\Pi_l|=6$,
$h-1=11$, $h^*-2=10$. On the other hand, $\mathfrak{X}$ has six
orbits on $\Delta(1)$, each has $11$ elements. Moreover, the size of
the lower ideals in each orbit is distributed as follows:
\begin{itemize}
\item[$\bullet$] $0, 1, 2, 4, 7, \textbf{10}, 13, 16, 18, 19, 20$;

\item[$\bullet$] $3, 4, 5, 6, 9, \textbf{10}, 11, 14, 15, 16, 17$;

\item[$\bullet$] $3, 4, 5, 6, 9, \textbf{10}, 11, 14, 15, 16, 17$;

\item[$\bullet$] $7, 7, 8, 8, 9, \textbf{10}, 11, 12, 12, 13, 13$;

\item[$\bullet$] $5, 6, 6, 8, 9, \textbf{10}, 11, 12, 14, 14, 15$;

\item[$\bullet$] $7, 7, 8, 8, 9, \textbf{10}, 11, 12, 12, 13, 13$.
\end{itemize}
One sees that each orbit has a unique Lagrangian lower ideal.

This finishes the proof of Theorem \ref{thm-main}. \hfill\qed

\medskip

\centerline{\scshape Acknowledgements} The research is supported by
the National Natural Science Foundation of China (grant no.
11571097) and the Fundamental Research Funds for the Central
Universities.

\end{document}